\documentclass[reqno]{amsart}
\usepackage{amssymb}
\usepackage{graphicx}

\usepackage[usenames, dvipsnames]{color}
\usepackage{verbatim}
\usepackage{mathrsfs}
\usepackage{bm}
\usepackage{cite}

\numberwithin{equation}{section}

\newtheorem{theorem}{Theorem}[section]
\newtheorem{corollary}[theorem]{Corollary}
\newtheorem{lemma}[theorem]{Lemma}

\theoremstyle{definition}
\newtheorem{remark}[theorem]{Remark}

\theoremstyle{definition}
\newtheorem{definition}[theorem]{Definition}

\theoremstyle{definition}

\makeatletter
\def\dashint{\operatorname%
{\,\,\text{\bf-}\kern-.98em\DOTSI\intop\ilimits@\!\!}}
\makeatother

\def\\det{\text{det}}

\def\.5{\frac{1}{2}}

\newcommand{\RN}[1]{%
  \textup{\uppercase\expandafter{\romannumeral#1}}%
}

\renewcommand{\epsilon}{\varepsilon}

\newcounter{marnote}

\begin{document}
\title[Asymptotic analysis for elliptic systems ]{Asymptotic analysis for elliptic systems in a narrow region}
\author[Z.W. Zhao]{Zhiwen Zhao}

\address[Z.W. Zhao]{1. Beijing Computational Science Research Center, Beijing 100193, China.}
\address{2. School of Mathematical Sciences, Beijing Normal University, Beijing 100875, China.}
\email{zwzhao365@163.com}

%\footnote{}

\date{\today} % delete this line to display the current date

%%% BEGIN DOCUMENT

%\tableofcontents

\begin{abstract}
This paper is concerned with investigating the asymptotic behavior of the gradients of solutions to a class of elliptic systems with general boundary data, especially covering the Lam\'{e} systems, in a narrow region. The novelty of this paper lies in finding the correction term to improve the previous gradient estimates to be a precise asymptotic formula. Through this justification, we capture all singular terms and thus give a fairly sharp characterization in terms of the singularities of the gradients.
\end{abstract}

\maketitle

\noindent{\bf{Keywords}}: Asymptotic expansions; correction term; elliptic systems; composites.

\section{Introduction and main results}
This work is devoted to the investigation of high concentration phenomena appearing in fiber-reinforced composites, where a relatively large number of fibers are touching or close to touching. This high concentration will occur in the thin gaps between two adjacent fibers or the narrow regions between the fibers and the matrix boundary and its singularity can be quantitatively described by the distance between two fibers or between the fibers and the external boundary. The principal quantity of interest from an engineering point of view is the stress, which is the gradient of a solution to the linear systems of elasticity, also called the Lam\'{e} systems, see \cite{OSY1992}. For further generalizations, we investigate a class of elliptic systems. The objectives of this paper are two-fold: on one hand, by justifying the leading term used in \cite{JLX2019}, we establish an asymptotic formula of the gradient for the elliptic systems with no-zero boundary value data; on the other hand, when the boundary value data degenerates to zero, we extend the exponentially decaying estimation of the gradient for $2$-convex inclusions in \cite{LLBY2014} to that for more general $m$-convex inclusions with $m\geq2$. In addition, for the convenience of applications, we list the similar results for the Lam\'{e} systems.

Motivated by the well-known work of Babu\u{s}ka et al. \cite{BASL1999} on initiation and growth of damage in composites, where they observed numerically that the size of the strain tensor is bounded independently of the distance $\varepsilon$ between inclusions, there has been a long list of papers on gradient estimates for solutions of elliptic equations and systems with piecewise smooth coefficients. Li and Nirenberg \cite{LN2003} demonstrated the observation in \cite{BASL1999} and established the $C^{1,\alpha}$ estimates for solutions to general second order elliptic systems, including systems of linear elasticity, with piecewise H\"{o}lder continuous coefficients in all dimensions. The corresponding results for divergence form elliptic equations can be seen in \cite{BV2000,LV2000}. Notice that the dependence on the elliptic coefficients in the upper bounds of the gradients in \cite{LV2000,LN2003,BV2000} is not explicit. Recently Dong and Li \cite{DL2019} showed the clear dependence of the elliptic coefficients for a class of non-homogeneous elliptic equations with discontinuous coefficients in the presence of two circular fibers in two dimensions. However, for more general elliptic equations and systems and more general shape of inclusions, this problem is still unsolved. We draw the attention of readers to more open problems in page 94 of \cite{LV2000} and page 894 of \cite{LN2003}.

It is well known that for the scalar case, the antiplane shear model is consistent with the two-dimensional conductivity model. So, analysing and understanding the singularities of the electric filed has a valuable meaning in relation to failure analysis of composites. It has been proved that the electric field, which is the gradient of a solution to the perfect conductivity equation, blows up at the rate of $\varepsilon^{-1/2}$ in two dimensions \cite{AKLLL2007,BC1984,BLY2009,AKL2005,Y2007,Y2009,K1993}, $|\varepsilon\ln\varepsilon|^{-1}$ in three dimensions \cite{BLY2009,LY2009,BLY2010,L2012}, and $\varepsilon^{-1}$ in higher dimensions \cite{BLY2009}, respectively. Besides those upper and lower bound estimates of the concentrated field, many mathematicians used different methods to establish a precise asymptotic formula for the field, for example, \cite{KLY2013,ACKLY2013,KLY2014,LLY2019}. We here would like to point out that some techniques used in the scalar equation, such as the maximum principle, cannot apply to the elliptic systems. This fact prevented us from extending the results for scalar equations to systems of equation until Li, Li, Bao, Yin \cite{LLBY2014} developed a delicate iteration technique to overcome this difficulty and established $C^{k}$ estimates for a class of elliptic systems for every $k\geq0$. Thenceforward this iteration technique is always used to investigate the gradient blow-up for the Lam\'{e} system, see \cite{BLL2015,BLL2017,BJL2017,LZ2020} and references therein. It is worthwhile to mention that Kang and Yu \cite{KY2019} recently utilized the layer potential techniques and the variational principle to give a quantitative characterization of the stress and proved that its optimal blow-up rate is $\varepsilon^{-1/2}$ in two dimensions.

In addition, some literature, beginning with \cite{ADKL2007}, aim to reveal the effect of the boundary data on the stress blow-up. Ammari et al. \cite{ADKL2007} studied the perfect conductivity problem with zero boundary value data as follows:
\begin{align*}
\begin{cases}
\Delta u=0,&\hbox{in}\ \mathbb{R}^3\setminus\overline{B_1\cup B_2},\\
u=0, &\hbox{on}\ \partial B_1\cup \partial B_2,\\
u(x)-H(x)=O(|x|^{-1}), &\hbox{as}\ |x|\rightarrow+\infty,
\end{cases}
\end{align*}
where $B_1$ and $B_2$ are two balls in $\mathbb{R}^3$ and $H$ is a harmonic function in $\mathbb{R}^3$ satisfying $H(0)=0$. Their conclusion is that there exists a constant $C$ independent of $\varepsilon$ such that $$\|\nabla(u-H)\|_{L^\infty(\mathbb{R}^3\setminus\overline{B_1\cup B_2})}\leq C,$$
which implies that there is no blow-up appearing. This result was extended to a class of elliptic systems in all dimensions in \cite{LLBY2014} and was improved that $|\nabla u|$ decays exponentially fast near the origin. Ju, Li and Xu \cite{JLX2019} further obtained the optimal gradient estimates for the elliptic systems with general no-zero boundary data. Additionally, in the context of the Lam\'{e} systems with partially infinite coefficients, Li and Zhao \cite{LZ2020} found that some special boundary data of $k$-order growth can strengthen the singularities of the stress.

To formulate our problem, we first fix our domain. For the sake of convenience, we use superscript prime to denote $(n-1)$-dimensional domains and variables throughout the paper. For example, for a small positive constant $R_{0}$, independent of $\varepsilon$, we use $B_{2R_{0}}'$ to denote a ball in $\mathbb{R}^{n-1}$ centered at the origin $0'$ of radius $2R_{0}$. Let $h_{1}$ and $h_{2}$ be two smooth functions in $B_{2R_{0}}'$ such that for $m\geq2$,
\begin{enumerate}
{\it\item[(\bf{A1})]
$\kappa_{1}|x'|^{m}\leq h_{1}(x')-h_{2}(x')\leq \kappa_{2}|x'|^{m},\;\mbox{if}\;\,x'\in B'_{2R_{0}},$
\item[(\bf{A2})]
$|\nabla_{x'}^{j}h_{i}(x')|\leq \kappa_{3}|x'|^{m-j},\;\mbox{if}\;\,x'\in B_{2R_{0}}',\;i,j=1,2,$
\item[(\bf{A3})]
$\|h_{1}\|_{C^{2}(B'_{2R_{0}})}+\|h_{2}\|_{C^{2}(B'_{2R_{0}})}\leq \kappa_{4},$}
\end{enumerate}
where $\kappa_{i},i=1,2,3,4$, are four positive constants independent of $\varepsilon$. For $z'\in B'_{R_{0}},\,0<t\leq2R_{0}$, we define
\begin{align*}
\Omega_{t}(z'):=&\left\{x\in \mathbb{R}^{d}~\big|~h_{2}(x')<x_{n}<\varepsilon+h_{1}(x'),~|x'-z'|<{t}\right\}.
\end{align*}
We adopt the abbreviated notation $\Omega_{t}$ for the domain $\Omega_{t}(0')$. For $0<r\leq2R_{0}$, the top and bottom boundaries of $\Omega_r$ can be represented by
$$\Gamma_r^+=\{x\in \mathbb{R}^n\, |\, x_{n}=\varepsilon+h_1(x'), |x'|\leq r\},~ \Gamma_r^-=\{x\in \mathbb{R}^n \,|\, x_{n}=h_2(x'), |x'|\leq r\},$$
respectively.
\begin{figure}[htb]
\center{\includegraphics[width=0.45\textwidth]{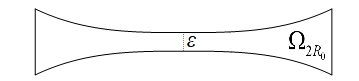}}
\caption{A narrow region}
\end{figure}

Let $u=(u^1,...,u^N)$ be a vector-valued function. In this paper, we consider the following boundary value problem in the narrow region (see Figure 1):
\begin{align}\label{eq1.1}
\begin{cases}
\partial_\alpha\left(A_{ij}^{\alpha\beta}(x)\partial_\beta u^j+B_{ij}^\alpha(x) u^j\right)+C_{ij}^\beta(x) \partial_\beta u^j+D_{ij}(x)u^j=0,\quad&
\hbox{in}\  \Omega_{2R_{0}},  \\
u=\varphi(x),\quad &\hbox{on}\ \Gamma_{2R_{0}}^+,\\
u=\psi(x), \quad&\hbox{on} \ \Gamma_{2R_{0}}^-,
\end{cases}
\end{align}
where $\varphi(x)=(\varphi^1(x), \varphi^2(x),..., \varphi^N(x))\in C^2(\Gamma_{2R_{0}}^+; \mathbb{R}^N),\;\psi(x)=(\psi^1(x),\;\psi^2(x),..., $
$\psi^N(x))\in C^2(\Gamma_{2R_{0}}^-; \mathbb{R}^N)$ are given vector-valued functions. Here and below the usual summation convention is used: $\alpha$ and $\beta$ are summed from 1 to $n$, while $i$ and $j$
are summed from 1 to $N$.

The coefficients  $A_{ij}^{\alpha\beta}(x)$ are measurable and bounded, that is,
\begin{align}\label{eq1.2}
|A_{ij}^{\alpha\beta}(x)| \leq \Lambda,\quad0<\Lambda<\infty,
\end{align}
and satisfy the weak ellipticity condition, that is,
there exists a constant $0<\lambda<\infty$ such that
\begin{align}\label{eq1.3}
\int_{\Omega_{2R_{0}}}A_{ij}^{\alpha\beta}(x)\partial_{\alpha}\xi^i\partial_{\beta}\xi^jdx\geq \lambda\int_{\Omega_{2R_{0}}}|\nabla\xi|^2dx,\quad \forall\ \xi\in H_0^1(\Omega_{2R_{0}}; \mathbb{R}^N ).
\end{align}
Moreover, when we fix $\alpha=\beta=n$, the matrix $(A_{ij}^{nn}(x))_{N\times N}$ satisfies the uniform elliptic condition:
\begin{align}\label{QDZAQ}
\Lambda_{1}|\eta|^{2}\leq A^{nn}_{ij}(x)\eta_{i}\eta_{j}\leq\Lambda_{2}|\eta|^{2}\quad\forall\;\eta\in\mathbb{R}^{N},
\end{align}
where $0<\Lambda_{1}\leq\Lambda_{2}<\infty$. In fact, in view of \eqref{eq1.2}, we may choose $\Lambda_{2}=N^{2}\Lambda$. With regard to the coefficients, we further assume that \begin{align}\label{ZKL987}\|A\|_{C^2(\Omega_{2R_{0}})}+\|B\|_{C^2(\Omega_{2R_{0}})}+\|C\|_{C^2(\Omega_{2R_{0}})}+\|D\|_{C^2(\Omega_{2R_{0}})}\leq \tau_{0},\end{align}
for some positive constant $\tau_{0}$.

Introduce a scalar auxiliary function $\bar{v}\in C^{2}(\mathbb{R}^{d})$ such that
\begin{align*}
\bar{v}(x',x_{n})=\frac{x_{n}-h_{2}(x')}{\delta(x')},\quad\mathrm{in}\;\Omega_{2R_{0}},
\end{align*}
where
\begin{align*}
\delta(x'):=\varepsilon+h_{1}(x')-h_{2}(x').
\end{align*}
For $(x',x_{n})\in\Omega_{2R_{0}}$, define
\begin{align}\label{OPQ}
\bar{u}(x',x_{n})=\varphi(x',\varepsilon+h_{1}(x'))\bar{v}+\psi(x',h_{2}(x'))(1-\bar{v})+\mathfrak{r}(\bar{v})\sum^{N}_{l=1}\mathcal{G}_{l},
\end{align}
where
\begin{align}\label{ACZZW01}
\mathfrak{r}(\bar{v}):=\frac{1}{2}\left(\bar{v}-\frac{1}{2}\right)^{2}-\frac{1}{8},
\end{align}
and for every $l=1,2,...,N$, $\mathcal{G}_{l}:=(g_{l}^{1},g_{l}^{2},...,g_{l}^{N})$ solves the following linear system of equations: for $i=1,2,...,N$,
\begin{align}\label{ZZW666}
A_{ij}^{nn}g_{l}^{j}=\Big(\sum^{n-1}_{\alpha=1}A_{il}^{\alpha n}\partial_{\alpha}\delta+\sum^{n-1}_{\beta=1}A_{il}^{n\beta}\partial_{\beta}\delta\Big)(\varphi^{l}(x',\varepsilon+h_{1}(x'))-\psi^{l}(x',h_{2}(x'))).
\end{align}
We would like to point out that hypothetical condition \eqref{QDZAQ} is critical to the determination of $\mathcal{G}_{l}$. In fact, it follows from the uniform elliptic condition \eqref{QDZAQ} that the matrix $(A^{nn}_{ij})_{N\times N}$ is positive definite and thus  invertible. Therefore, by using the Cramer's rule, we know that $\mathcal{G}_{l}$ is uniquely determined by the linear system of equations \eqref{ZZW666}.

\begin{definition}
We say that $u\in{H}^{1}(\Omega_{2R_{0}};\mathbb{R}^N)$ is a weak solution of problem (\ref{eq1.1}) if
$$\int_{\Omega_{2R_{0}}}\left(A_{ij}^{\alpha\beta}(x)\partial_\beta u^j+B_{ij}^\alpha(x) u^j\right)\partial_{\alpha}\zeta^{i}-C_{ij}^\beta(x) \partial_\beta u^j\zeta^{i}-D_{ij}(x)u^j\zeta^{i}dx=0,$$
for every vector-valued function $\zeta=(\zeta^{1},...,\zeta^{N})\in{C}_{c}^{\infty}(\Omega_{2R_{0}};\mathbb{R}^N)$, and hence for every $\zeta\in{H}_{0}^{1}(\Omega_{2R_{0}}; \mathbb{R}^N)$.
\end{definition}
For the convenience of presentation, we denote
\begin{align}\label{QWA019}
\Theta(\varphi,\psi)(x'):=&|\varphi(x',\varepsilon+h_{1}(x'))-\psi(x',h_{2}(x'))|\notag\\
&+|\nabla_{x'}(\varphi(x',\varepsilon+h_{1}(x'))-\psi(x',h_{2}(x')))|.
\end{align}

Throughout this paper, unless otherwise stated, we use $C$ to denote some positive constant, whose values may vary from line to line, which depends only on $n,N,\ \lambda,\ \Lambda, \ R_{0}, \ \tau_{0}, \ \kappa_{1},\ \kappa_{2},\ \kappa_{3},\ \kappa_{4}$, but not on $\varepsilon$. $O(1)$ denotes some quantity satisfying $|O(1)|\leq C$ for some $\varepsilon$-independent constant C.
\begin{theorem}\label{thm1.1}
Assume that hypotheses $\rm{(}${\bf{A1}}$\rm{)}$--$\rm{(}${\bf{A3}}$\rm{)}$ hold, and $\varphi\neq0$ on $\Gamma^{+}_{2R_{0}}$ or $\psi\neq0$ on $\Gamma^{-}_{2R_{0}}$. Let $u\in H^{1}(\Omega_{2R_{0}}; \mathbb{R}^N)$ be a weak solution of problem (\ref{eq1.1}). Then, for a sufficiently small $\varepsilon>0$, $x\in \Omega_{R_{0}}$,
\begin{align}\label{DZA001}
\nabla u=&\nabla\bar{u}+O(1)\Big(\Theta(\varphi,\psi)(x')+\delta\big(\|\varphi\|_{C^{2}(\Gamma^{+}_{2R_{0}})}+\|\psi\|_{C^{2}(\Gamma^{-}_{2R_{0}})}\big)\Big),
\end{align}
where the leading term $\bar{u}$ is given by \eqref{OPQ} and $\Theta(\varphi,\psi)(x')$ is defined by \eqref{QWA019}.
\end{theorem}

\begin{remark}
It was seen in Lemma 2.3 of \cite{JLX2019} that $|\nabla(u-\varphi(x',\varepsilon+h_{1}(x'))\bar{v}-\psi(x',\varepsilon+h_{1}(x'))(1-\bar{v}))|\leq C|\varphi(x',\varepsilon+h_{1}(x'))-\psi(x',h_{2}(x'))|\varepsilon^{-1/2}$, which implies that $\nabla(\varphi(x',\varepsilon+h_{1}(x'))\bar{v}-\psi(x',\varepsilon+h_{1}(x'))(1-\bar{v}))$, the leading term used in \cite{JLX2019}, is not sufficiently precise. By contrast, we here capture its correction term $\nabla(\mathfrak{r}(\bar{v})\sum^{N}_{l=1}\mathcal{G}_{l})$, which makes the difference between $\nabla u$ and our corrected main term $\nabla\bar{u}$ be of order $O(1)$ in Theorem \ref{thm1.1}. Moreover, if $\Theta(\varphi,\psi)(0')=0$, then the remainder in \eqref{DZA001} becomes an infinitely small quantity of order $O(\varepsilon)$ in the shortest segment $\{(0',x_{n})|\,0<x_{n}<\varepsilon\}$.
\end{remark}

\begin{remark}
From the asymptotic \eqref{DZA001}, we see that the leading term $\nabla\bar{u}$ consists of two parts: one is higher order singular term $\nabla(\varphi(x',\varepsilon+h_{1}(x'))\bar{v}+\psi(x',h_{2}(x'))(1-\bar{v}))$ with its singularity of order $|\varphi(x',\varepsilon+h_{1}(x'))-\psi(x',h_{2}(x'))|\delta^{-1}$; the other part is the correction term $\nabla\big(\mathfrak{r}(\bar{v})\sum^{N}_{l=1}\mathcal{G}_{l}\big)$ and possesses the lower order singularity $|\varphi(x',\varepsilon+h_{1}(x'))-\psi(x',h_{2}(x'))|\delta^{-1/m}$. That is, the singularity of the gradient is determined by $|\partial_{n}\bar{u}|=|\varphi(x',\varepsilon+h_{1}(x'))-\psi(x',h_{2}(x'))|\delta^{-1}(1+O(1)\delta^{1-1/m})$. Then for $|x'|\leq2R_{0}$,

$(i)$ if $\varphi(x',\varepsilon+h_{1}(x'))-\psi(x',h_{2}(x'))=\mathbf{0}$, then \eqref{DZA001} becomes $\nabla u=\nabla (\varphi(x',\varepsilon+h_{1}(x')))+O(1)\delta\|\varphi\|_{C^{2}(\Gamma^{+}_{2R_{0}})}$, which indicates that $|\nabla u|$ remains bounded;

$(ii)$ if $\varphi(x',\varepsilon+h_{1}(x'))-\psi(x',h_{2}(x'))=\mathfrak{a}$ for some constant vector $\mathfrak{a}\neq\mathbf{0}$, then $|\nabla u|$ blows up at the rate of $\varepsilon^{-1}$ in the shortest segment $\{(0',x_{n})|\,0<x_{n}<\varepsilon\}$;

$(iii)$ if $\varphi(x',\varepsilon+h_{1}(x'))-\psi(x',h_{2}(x'))=(x_{1}^{k},0,...,0)$ for a positive integer $k\geq1$, then $|\nabla u|$ exhibits the singularity of order $O(\varepsilon^{k/m-1})$ at the point $x'=(\varepsilon^{1/m},0,...,0)$ provided that $m>k$.
\end{remark}

\begin{remark}
We would like to point out that our results can be extended to the case when the upper and lower boundaries of a narrow region are touching at one point and to a narrow region with partially flat upper and lower boundaries.
\end{remark}

When the boundary data degenerates to zero, that is, $\varphi=0$ on $\Gamma^{+}_{2R_{0}}$ and $\psi=0$ on $\Gamma^{-}_{2R_{0}}$, the elliptic systems \eqref{eq1.1} becomes
\begin{align}\label{eqation01}
\begin{cases}
\partial_\alpha\left(A_{ij}^{\alpha\beta}(x)\partial_\beta u^j+B_{ij}^\alpha(x) u^j\right)+C_{ij}^\beta(x) \partial_\beta u^j+D_{ij}(x)u^j=0,&
\hbox{in}\;\Omega_{2R_{0}},  \\
u=0,&\hbox{on}\;\Gamma_{2R_{0}}^{\pm}.
\end{cases}
\end{align}

Then, we have
\begin{theorem}\label{thm003}
Assume that hypotheses $\rm{(}${\bf{A1}}$\rm{)}$--$\rm{(}${\bf{A3}}$\rm{)}$ hold, $\varphi=0$ on $\Gamma^{+}_{2R_{0}}$ and $\psi=0$ on $\Gamma^{-}_{2R_{0}}$. Let $u\in H^{1}(\Omega_{2R_{0}}; \mathbb{R}^N)$ be a weak solution of problem (\ref{eqation01}). Then, for a sufficiently small $\varepsilon>0$, $x\in \Omega_{R_{0}}$,
\begin{align*}
\nabla u=O(1)\delta^{-\frac{n}{2}}e^{-\frac{1}{2C\delta^{1-1/m}}}\|u\|_{L^{2}(\Omega_{2R_{0}})}.
\end{align*}
In particular,
\begin{align*}
\max_{0<x_{n}<\varepsilon}|\nabla u(0',x_{n})|=O(1)\varepsilon^{-\frac{n}{2}}e^{-\frac{1}{2C\varepsilon^{1-1/m}}}\|u\|_{L^{2}(\Omega_{2R_{0}})}\rightarrow0,\quad\mathrm{as}\;\varepsilon\rightarrow0.
\end{align*}
\end{theorem}

\begin{remark}
Assume further that $A_{ij}^{\alpha\beta}$, $B_{ij}^{\alpha}$, $C_{ij}^{\beta}$, $D_{ij}$, $h_{1}$ and $h_{2}$ belong to $C^{m}(\Omega_{2R_{0}})$ for $m\geq2$, that is,
\begin{align*}
\|A\|_{C^m(\Omega_{2R_{0}})}+\|B\|_{C^m(\Omega_{2R_{0}})}+\|C\|_{C^m(\Omega_{2R_{0}})}+\|D\|_{C^m(\Omega_{2R_{0}})}\leq \sigma_{m},
\end{align*}
and
\begin{align*}
\|h_{1}\|_{C^{m}(\Omega_{2R_{0}})}+\|h_{2}\|_{C^{m}(\Omega_{2R_{0}})}\leq\varsigma_{m},
\end{align*}
where $\sigma_{m}$ and $\varsigma_{m}$ are two positive constants. Then almost the same to the proof of Theorem \ref{thm003}, we can establish $C^{m}$ decay estimates for solutions of elliptic systems \eqref{eqation01} as follows:
\begin{align*}
|\nabla^{m}u(x)|\leq C\delta^{1-m-\frac{n}{2}}e^{-\frac{1}{2C\delta^{1-1/m}}}\|u\|_{L^{2}(\Omega_{2R_{0}})},\quad x\in\Omega_{R_{0}}.
\end{align*}

\end{remark}

The rest of this paper is organized as follows. By utilizing the corrected leading term introduced in \eqref{OPQ} and the iteration technique with respect to the energy developed in \cite{LLBY2014}, we give the proofs of Theorems \ref{thm1.1} and \ref{thm003} in Sections \ref{sec_2} and \ref{sc3}, respectively. A direct application in the Lam\'{e} systems is presented in Section \ref{sc56}.

\section{Proof of Theorem \ref{thm1.1}}\label{sec_2}

We decompose the solution of (\ref{eq1.1}) as follows:
\begin{align*}
u=v_1+v_2+\cdots+v_N,
\end{align*}
where $v_l=(v_l^1, v_l^2,..., v_l^N)$, $l=1,2,..., N$, with  $v_l^j=0$ for $j\neq l$, and $v_l$ satisfies the following boundary value problem
\begin{align}\label{eq_v2.1}
\begin{cases}
  \partial_\alpha\left(A_{ij}^{\alpha\beta}(x)\partial_\beta v_l^j+B_{ij}^\alpha(x) v_l^j\right)+C_{ij}^\beta(x) \partial_\beta v_l^j+D_{ij}(x)v_l^j=0,
&\hbox{in}\  \Omega_{2R_{0}},  \\
v_l=(0,..., 0, \varphi^l, 0,...,0),\ &\hbox{on}\ \Gamma_{2R_{0}}^+,\\
v_l=(0,..., 0, \psi^l, 0,...,0),&\hbox{on} \ \Gamma_{2R_{0}}^-.
\end{cases}
\end{align}
Then
\begin{equation}\label{split001}%\label{equ_nablau}
\nabla{u}=\sum_{l=1}^{N}\nabla{v}_{l}.
\end{equation}
Define the basis of $\mathbb{R}^{N}$ as follows:
\begin{align*}
e_{l}=(0,...,0,1,0,...,0),\quad l=1,2,...,N.
\end{align*}
For $l=1,2,...,N$, denote
\begin{equation}\label{equ_tildeu}
\bar{u}_{l}(x)=\left(\varphi^{l}(x',\varepsilon+h_{1}(x'))\bar{v}(x)+\psi^{l}(x',h_{2}(x'))(1-\bar{v}(x))\right)e_{l}+\mathfrak{r}(\bar{v})\mathcal{G}_{l},
\end{equation}
where $\mathfrak{r}(\bar{v})$ and $\mathcal{G}_{l}$ are given by \eqref{ACZZW01}--\eqref{ZZW666}, respectively. Then, we see
\begin{align}\label{split002}
\bar{u}=\sum^{N}_{l=1}\bar{u}_{l}.
\end{align}
By a direct calculation, we obtain
\begin{align}\label{AZQ001}
|\nabla_{x'}\bar{v}(x)|\leq\frac{1}{\delta^{1/m}},\quad\partial_{n}\bar{v}(x)=\frac{1}{\delta},\quad\partial_{\alpha n}\bar{v}=-\frac{\partial_{\alpha}\delta}{\delta^{2}},\quad\alpha=1,...,n-1.
\end{align}
Thus,
\begin{align}
|\nabla_{x'}\bar{u}_l(x)|\leq&\frac{C|\varphi^{l}(x',\varepsilon+h_{1}(x'))-\psi^{l}(x',h_{2}(x'))|}{\delta^{1/m}}\notag\\
&+C(\|\varphi^{l}\|_{C^{1}}+\|\psi^{l}\|_{C^{1}}),\label{AZQ002}\\
|\partial_{n}\bar{u}_l(x)|\leq&\frac{C|\varphi^l(x',\varepsilon+h_{1}(x'))
-\psi^l(x',h_{2}(x'))|}{\delta},\label{AZQ003}
\end{align}
and
\begin{align}
|\nabla_{x'}^{2}\bar{u}_l(x)|\leq&\frac{C|\varphi^{l}(x',\varepsilon+h_{1}(x'))-\psi^{l}(x',h_{2}(x'))|}{\delta^{2/m}}+C(\|\varphi^{l}\|_{C^{2}}+\|\psi^{l}\|_{C^{2}})\notag\\
&+\frac{C(|\nabla_{x'}(\varphi^{l}(x',\varepsilon+h_{1}(x'))-\psi^{l}(x',h_{2}(x')))|}{\delta^{1/m}},\label{AZQ004}\\
\partial_{nn}\bar{u}_{l}=&\frac{\mathcal{G}_{l}}{\delta^{2}}+\mathfrak{r}(\bar{v})\partial_{nn}\mathcal{G}_{l}.\label{AZQ005}
\end{align}
Here and below, for $j=1,2$, we use $\|\varphi^{l}\|_{C^{j}}$ and $\|\psi^{l}\|_{C^{j}}$ to denote $\|\varphi^{l}\|_{C^{j}(\Gamma_{2R_{0}}^{+})}$ and $\|\psi^{l}\|_{C^{j}(\Gamma_{2R_{0}}^{-})}$, respectively.

Let $$w_l=v_l-\bar{u}_l,\qquad l=1,2,...,N.$$
Then $w_{l}$ satisfies
\begin{align}\label{eq2.6}
\begin{cases}
\partial_\alpha\big(A_{ij}^{\alpha\beta}(x)\partial_\beta w^j_{l}+B_{ij}^\alpha(x) w^j_{l}\big)+C_{ij}^\beta(x) \partial_\beta w^j_{l}+D_{ij}(x)w^j_{l}=f^{i}_{l},&
\hbox{in}\;\Omega_{2R_{0}},  \\
w_{l}=0,&\hbox{on}\;\Gamma_{2R_{0}}^\pm,
\end{cases}
\end{align}
where
\begin{align}\label{ZWPLN001}
f^{i}_{l}=&
-\left(\partial_{\alpha}A_{ij}^{\alpha\beta}(x)\partial_{\beta}\bar{u}_{l}^j+\partial_{\alpha}(B_{ij}^{\alpha}(x)\bar{u}_{l}^{j})+C_{ij}^{\beta}(x)\partial_\beta\bar{u}_{l}^{j}+D_{ij}(x)\bar{u}_{l}^{j}\right)-A_{ij}^{\alpha\beta}(x)\partial_{\alpha\beta}\bar{u}_{l}^j\notag\\
=&-\mathrm{I}-\mathrm{II}.
\end{align}
On one hand, from \eqref{ZKL987} and \eqref{AZQ002}--\eqref{AZQ003}, we have
\begin{align*}
|\mathrm{I}|\leq C|\varphi^l(x',\varepsilon+h_{1}(x'))
-\psi^l(x',h_{2}(x'))|\delta^{-1}+C(\|\varphi^{l}\|_{C^{1}}+\|\psi^{l}\|_{C^{1}}).
\end{align*}

On the other hand, recalling the definition of $\bar{u}_{l}$ in \eqref{equ_tildeu} and making use of \eqref{AZQ005}, we have
\begin{align*}
\mathrm{II}=&\sum^{n-1}_{\alpha,\beta=1}A_{ij}^{\alpha\beta}(x)\partial_{\alpha\beta}\bar{u}_{l}^j+\sum^{n-1}_{\alpha=1}A_{ij}^{\alpha n}(x)\partial_{\alpha n}\bar{u}_{l}^j+\sum^{n-1}_{\beta=1}A_{ij}^{n\beta}(x)\partial_{n\beta}\bar{u}_{l}^j+A_{ij}^{nn}\partial_{nn}\bar{u}_{l}^{j}\\
=&\Bigg(\sum^{n-1}_{\alpha,\beta=1}A_{ij}^{\alpha\beta}(x)\partial_{\alpha\beta}\bar{u}_{l}^j+\sum^{n-1}_{\alpha=1}A_{il}^{\alpha n}\frac{\partial_{\alpha}(\varphi^{l}(x',\varepsilon+h_{1}(x'))-\psi^{l}(x',h_{2}(x')))}{\delta}\notag\\
&+\sum^{n-1}_{\alpha=1}A_{ij}^{\alpha n}\partial_{\alpha n}(\mathfrak{r}(\bar{v})g_{l}^{j})+\sum^{n-1}_{\beta=1}A_{il}^{n\beta}\frac{\partial_{\beta}(\varphi^{l}(x',\varepsilon+h_{1}(x'))-\psi^{l}(x',h_{2}(x')))}{\delta}\notag\\
&+\sum^{n-1}_{\beta=1}A_{ij}^{n\beta}\partial_{n\beta}(\mathfrak{r}(\bar{v})g_{l}^{j})+2A_{ij}^{nn}\partial_{n}\mathfrak{r}(\bar{v})\partial_{n}g_{l}^{j}+\mathfrak{r}(\bar{v})A_{ij}^{nn}\partial_{nn}g_{l}^{j}\Bigg)\notag\\
&+\frac{1}{\delta^{2}}\Bigg(A_{ij}^{nn}g_{l}^{j}-\bigg(\sum^{n-1}_{\alpha=1}A_{il}^{\alpha n}\partial_{\alpha}\delta+\sum^{n-1}_{\beta=1}A_{il}^{n\beta}\partial_{\beta}\delta\bigg)\big(\varphi^{l}(x',\varepsilon+h_{1}(x'))-\psi^{l}(x',h_{2}(x'))\big)\Bigg)\\
=&\mathrm{II}_{1}+\mathrm{II}_{2}.
\end{align*}
Using \eqref{AZQ001}--\eqref{AZQ004}, we obtain
\begin{align*}
|\mathrm{II}_{1}|\leq&C|\nabla_{x'}(\varphi^{l}(x',\varepsilon+h_{1}(x'))-\psi^{l}(x',h_{2}(x')))|\delta^{-1}+C(\|\varphi^{l}\|_{C^{2}}+\|\psi^{l}\|_{C^{2}})\notag\\
&+C|\varphi^{l}(x',\varepsilon+h_{1}(x'))-\psi^{l}(x',h_{2}(x'))|\delta^{-2/m},
\end{align*}
while, in view of \eqref{ZZW666}, we have
\begin{align*}
\mathrm{II}_{2}=0.
\end{align*}
For simplicity, we also use the notation $\Theta(\varphi,\psi)(x')$ introduced in \eqref{QWA019} to denote
\begin{align*}
|\varphi^{l}(x',\varepsilon+h_{1}(x'))-\psi^{l}(x',h_{2}(x'))|+|\nabla_{x'}(\varphi^{l}(x',\varepsilon+h_{1}(x'))-\psi^{l}(x',h_{2}(x')))|.
\end{align*}
Then
\begin{align}\label{KHQ010}
|f^{i}_{l}|\leq&C\Theta(\varphi,\psi)(x')\delta^{-1}+C(\|\varphi^{l}\|_{C^{2}}+\|\psi^{l}\|_{C^{2}}).
\end{align}

\begin{lemma}\label{lem2.1}
Let $v_l\in H^1(\Omega_{2R_{0}}; \mathbb{R}^N)$ be a weak solution of (\ref{eq_v2.1}), then for $l=1,2,...,N$,
\begin{align}\label{lem2.2equ}
\int_{\Omega_{R_{0}}}|\nabla w_l|^2dx\leq C\big(\|w_l\|^2_{L^2(\Omega_{2R_{0}})}+\|\varphi^{l}\|_{C^{2}(\Gamma_{2R_{0}}^{+})}^{2}
+\|\psi^{l}\|_{C^{2}(\Gamma_{2R_{0}}^{-})}^{2}\big),
\end{align}
where $C$ depends on $n$, $\lambda$, $\kappa_i$, $i=0,1,2,3,4$, but not on $\varepsilon$.
\end{lemma}

\begin{proof} For simplicity, denote
\begin{align}\label{nota001}
w:=w_l,\;\;\bar{u}:=\bar{u}_l,\;\,\varphi:=\varphi^l,\;\,\psi:=\psi^{l}\;\,\mathrm{and}\;\,f^{i}:=f^{i}_{l}.
\end{align}
Then,
\begin{align*}
f^{i}
=&-\partial_\alpha \left(A_{ij}^{\alpha\beta}(x)\partial_\beta \bar{u}^j+B_{ij}^\alpha(x) \bar{u}^j+C_{ij}^\alpha(x) \bar{u}^j\right)\nonumber\\
&+\partial_\beta (C_{ij}^\beta(x))\bar{u}^j-D_{ij}(x)\bar{u}^j.
\end{align*}
Multiplying the equation in (\ref{eq2.6}) by $w$ and applying integration by parts in $\Omega_{R_{0}}$, we have
\begin{align*}
&\int_{\Omega_{R_{0}}}A_{ij}^{\alpha\beta}(x)\partial_\beta w^j\partial_\alpha w^idx
\\
=&-\int_{\Omega_{R_{0}}}B_{ij}^{\alpha}(x) w^j\partial_\alpha w^idx
+\int_{\Omega_{R_{0}}}C_{ij}^{\beta}(x)\partial_\beta w^j w^idx+\int_{\Omega_{R_{0}}}D_{ij}(x) w^jw^idx\\
&-\int_{\Omega_{R_{0}}}f^{i}w^idx+\int\limits_{\scriptstyle |x'|={R_{0}},\atop\scriptstyle
h_2(x')<x_{n}<\varepsilon+h_1(x')\hfill}\left(A_{ij}^{\alpha\beta}(x)\partial_\beta w^j+B_{ij}^\alpha(x) w^j\right)w^{i}\frac{x_{\alpha}}{R_{0}}dS.
\end{align*}
Using the weak ellipticity condition and the Cauchy inequality, we obtain
\begin{align}\label{equ1}
\lambda\int_{\Omega_{R_{0}}}|\nabla w|^2dx
\leq&\,\int_{\Omega_{R_{0}}}A_{ij}^{\alpha\beta}(x)\partial_\beta w^j\partial_\alpha w^idx\nonumber\\
\leq&\,\frac{\lambda}{4}\int_{\Omega_{R_{0}}}|\nabla w|^2dx+C \int_{\Omega_{R_{0}}}|w|^2dx+\left|\int_{\Omega_{R_{0}}}f^{i}w^idx\right|\nonumber\\
&+\int\limits_{\scriptstyle |x'|={R_{0}},\atop\scriptstyle
h_2(x')<x_{n}<\varepsilon+h_1(x')\hfill}C\left(|\nabla w|^2+|w|^2\right)dS.
\end{align}
Note that $w=0$ on $\Gamma_{2R_{0}}^{\pm}$ and $\overline{\Omega_{4/3R_{0}}}\setminus \Omega_{2/3R_{0}}\subset\left((\Omega_{2R_{0}}\setminus\overline{\Omega_{1/2R_{0}}})\cup (\Gamma_{2R_{0}}^{\pm}\setminus \Gamma_{1/2R_{0}}^{\pm})\right)$,
 by using  the Sobolev embedding theorem and classical $W^{2, p}$ estimates for elliptic systems, % (see e.g. \cite{AD2}, or Theorem 2.5 in \cite{G}),
 we have, for some $p>n$,
\begin{align*}
\|\nabla w\|_{L^\infty(\Omega_{4/3R_{0}}\setminus \overline{\Omega_{2/3R_{0}}})}&\leq C\|w\|_{W^{2,p}(\Omega_{4/3R_{0}}\setminus \overline{\Omega_{2/3R_{0}}})}\\
&\leq C\big(\|w\|_{L^2(\Omega_{2R_{0}}\setminus \overline{\Omega_{1/2R_{0}}})}+\|f\|_{L^\infty(\Omega_{2R_{0}}\setminus \overline{\Omega_{1/2R_{0}}})}\big)\nonumber\\
&\leq C\big(\|w\|_{L^2(\Omega_{2R_{0}})}+\|\varphi\|_{C^2(\Gamma_{2R_{0}}^{+})}+\|\psi\|_{C^2(\Gamma_{2R_{0}}^{-})}\big),
\end{align*}
and for $x=(x',x_n)\in\Omega_{4/3R_{0}}\setminus \overline{\Omega_{2/3R_{0}}}$,
\begin{align*}
|w(x',x_n)|&=|w(x',x_n)-w(x',h_2(x'))|\\&\leq C\delta(x')\|\nabla w\|_{L^\infty(\Omega_{4/3R_{0}}\setminus  \overline{\Omega_{2/3R_{0}}})}\\
&\leq C\big(\|w\|_{L^2(\Omega_{2R_{0}})}+\|\varphi\|_{C^2(\Gamma_{2R_{0}}^{+})}+\|\psi\|_{C^2(\Gamma_{2R_{0}}^{-})}\big).
\end{align*}
In particular, this implies that
\begin{align}\label{w_on|x'|=r}
&\int\limits_{\scriptstyle |x'|={R_{0}},\atop\scriptstyle
h_2(x')<x_{n}<\varepsilon+h_1(x')\hfill}\big(| w|^{2}+|\nabla w|^{2}\big)dS\notag\\
&\leq C\big(\|w\|^2_{L^2(\Omega_{2R_{0}})}+\|\varphi\|^2_{C^2(\Gamma_{2R_{0}}^+)}+\|\psi\|^{2}_{C^2(\Gamma_{2R_{0}}^{-})}\big),
\end{align}
where $C$ depends only on $n,\,R_{0}$ and $\kappa_i$, $i=1,2,3,4$, but not on $\varepsilon$.

Then making use of \eqref{KHQ010} and the Cauchy's inequality, we have
\begin{align}\label{LFE01}
\left|\int_{\Omega_{R_{0}}}f^{i}w^idx\right|\leq&C\int_{\Omega_{R_{0}}}\frac{\Theta(\varphi,\psi)(x')}{\delta(x')}|w^{i}|dx+C(\|\varphi\|_{C^{2}}+\|\psi\|_{C^{2}})\int_{\Omega_{R_{0}}}|w^{i}|dx\notag\\
=&C\int_{|x'|<R_{0}/2}\frac{\Theta(\varphi,\psi)(x')}{\delta(x')}dx'\int^{\varepsilon+h_{1}}_{h_{2}}\left|\int^{x_{n}}_{h_{2}}\partial_{n}\omega^{i}(x',t)dt\right|dx_{n}\notag\\
&+C(\|\varphi\|_{C^{2}}+\|\psi\|_{C^{2}})\int_{\Omega_{R_{0}}}|w^{i}|dx\notag\\
\leq&C\int_{|x'|<R_{0}/2}\Theta(\varphi,\psi)(x')\left|\int^{x_{n}}_{h_{2}}\partial_{n}\omega^{i}(x',t)dt\right|dx'\notag\\
&+C(\|\varphi\|_{C^{2}}+\|\psi\|_{C^{2}})\int_{\Omega_{R_{0}}}|w^{i}|dx\notag\\
\leq&C(\|\varphi\|_{C^{1}}+\|\psi\|_{C^{1}})\int_{\Omega_{R_{0}}}|\nabla w|dx+C(\|\varphi\|_{C^{2}}+\|\psi\|_{C^{2}})\int_{\Omega_{R_{0}}}|w^{i}|dx\notag\\
\leq&\frac{\lambda}{4}\|\nabla w\|^{2}_{L^{2}(\Omega_{R_{0}})}+C\big(\|w\|^{2}_{L^{2}(\Omega_{R_{0}})}+\|\varphi\|^{2}_{C^{2}}+\|\psi\|^{2}_{C^{2}}\big).
\end{align}
Substituting \eqref{w_on|x'|=r}--\eqref{LFE01} into (\ref{equ1}), we see
\begin{align*}
\int_{\Omega_{R_{0}}}|\nabla w|^2dx\leq C\left(\|w\|^2_{L^2(\Omega_{2R_{0}})}+\|\varphi\|_{C^{2}(\Gamma_{2R_{0}}^{+})}^{2}+\|\psi\|_{C^{2}(\Gamma_{2R_{0}}^{-})}^{2}\right).
\end{align*}
The proof is complete.
\end{proof}

\begin{lemma}\label{lem2.2}
Assume as above. Then for a sufficiently small $\varepsilon>0$, $|z'|\leq R_{0}$, $l=1,2,...,N$,
\begin{align}\label{AVR001}
\int_{\Omega_{\delta}(z')}|\nabla w_{l}|^{2}&\leq C\delta^{n}\left(\Theta^{2}(\varphi,\psi)(z')+\delta^{2}\big(\|\varphi^{l}\|^{2}_{C^{2}(\Gamma^{+}_{2R_{0}})}+\|\psi^{l}\|^{2}_{C^{2}(\Gamma^{-}_{2R_{0}})}\big)\right).
\end{align}

\end{lemma}

\begin{proof}
For the sake of convenience, we still adopt the abbreviated notations as in \eqref{nota001}. For $0<t<s<1$,  let $\eta(x')$ be a smooth cutoff function such that $0\leq \eta(x')\leq1$, $\eta(x')=1$ if $|x'-z'|<t$, $\eta(x')=0$ if $|x'-z'|>s$ and $|\nabla\eta(x')|\leq\frac{2}{s-t}$. Multiplying the equation in (\ref{eq2.6}) by $\eta^2w^{i}$ and utilizing integration by parts, we obtain
\begin{align*}
\int_{\Omega_{s}(z')}f^{i}\eta^2w^idx=&-\int_{\Omega_{s}(z')}(A_{ij}^{\alpha\beta}(x)\partial_\beta w^j+B_{ij}^{\alpha}(x)w^j)\partial_\alpha(\eta^2w^i)dx\\
&+\int_{\Omega_{s}(z')}C_{ij}^{\beta}(x)\partial_\beta w^j\eta^2w^idx+\int_{\Omega_{s}(z')}D_{ij}(x) w^j\eta^2w^idx.
\end{align*}
Since\begin{align*}&\int_{\Omega_{s}(z')}(A_{ij}^{\alpha\beta}(x)\partial_\beta w^j+B_{ij}^{\alpha}(x)w^j)\partial_\alpha(\eta^2w^i)dx\\
=&\,\int_{\Omega_{s}(z')}A_{ij}^{\alpha\beta}(x)\partial_\beta(\eta w^j)\partial_\alpha(\eta w^i)dx
-\int_{\Omega_{s}(z')}A_{ij}^{\alpha\beta}(x)(\partial_\beta\eta w^j)\partial_\alpha(\eta w^i)dx\\
&+\int_{\Omega_{s}(z')}A_{ij}^{\alpha\beta}(x)\partial_\beta(\eta w^j)(\partial_\alpha\eta w^i)dx-\int_{\Omega_{s}(z')}A_{ij}^{\alpha\beta}(x)(\partial_\beta\eta w^j)(\partial_\alpha\eta w^i)dx\\
&+\int_{\Omega_{s}(z')}B_{ij}^{\alpha}(x)(\eta w^j)\partial_\alpha(\eta w^i)dx+\int_{\Omega_{s}(z')}B_{ij}^{\alpha}(x)(\partial_\alpha \eta w^j)(\eta w^i)dx,
\end{align*}
and
\begin{align*}
&\int_{\Omega_{s}(z')}C_{ij}^{\beta}(x)\partial_\beta w^j\eta^2w^idx\\
=&\,\int_{\Omega_{s}(z')}C_{ij}^{\beta}(x)\partial_\beta(\eta w^j)(\eta w^i)dx-\int_{\Omega_{s}(z')}C_{ij}^{\beta}(x)(\partial_\beta \eta w^j)(\eta w^i)dx,
\end{align*}
it follows from the weak ellipticity condition (\ref{eq1.3}) and the Cauchy inequality that
\begin{align*}&\lambda\int_{\Omega_{s}(z')}|\nabla(\eta w)|^2dx\leq\int_{\Omega_{s}(z')}A_{ij}^{\alpha\beta}(x)\partial_\beta(\eta w^j)\partial_\alpha(\eta w^i)dx\\
=&\,\int_{\Omega_{s}(z')}A_{ij}^{\alpha\beta}(x)(\partial_\beta\eta w^j)\partial_\alpha(\eta w^i)dx
-\int_{\Omega_{s}(z')}A_{ij}^{\alpha\beta}(x)\partial_\beta(\eta w^j)(\partial_\alpha\eta w^i)dx\\
&+\int_{\Omega_{s}(z')}A_{ij}^{\alpha\beta}(x)(\partial_\beta\eta w^j)(\partial_\alpha\eta w^i)dx
-\int_{\Omega_{s}(z')}B_{ij}^{\alpha}(x)(\eta w^j)\partial_\alpha(\eta w^i)dx\\
&-\int_{\Omega_{s}(z')}B_{ij}^{\alpha}(x)(\partial_\alpha \eta w^j)(\eta w^i)dx+\int_{\Omega_{s}(z')}C_{ij}^{\beta}(x)\partial_\beta(\eta w^j)(\eta w^i)dx\\
&-\int_{\Omega_{s}(z')}C_{ij}^{\beta}(x)(\partial_\beta \eta w^j)(\eta w^i)dx\nonumber +\int_{\Omega_{s}(z')}D_{ij}(x) (\eta w^j)(\eta w^i)dx-\int_{\Omega_{s}(z')}\eta^{2}f^{i}w^idx\\
\leq&\, \frac{\lambda}{2}\int_{\Omega_{s}(z')}|\nabla(\eta w)|^2dx+C\int_{\Omega_{s}(z')}|w\nabla\eta|^2dx+\frac{C}{(s-t)^2}\int_{\Omega_s(x_0)}|w|^2dx \\
&+(s-t)^2\int_{\Omega_s(x_0)}|f|^2dx\\
\leq&\, \frac{\lambda}{2}\int_{\Omega_{s}(z')}|\nabla(\eta w)|^2dx+\frac{C}{(s-t)^2}\int_{\Omega_{s}(z')}|w|^2dx+(s-t)^2\int_{\Omega_{s}(z')}|f|^2dx.
\end{align*}
So we obtain an iteration formula as follows:
\begin{align}\label{KTA001}
\int_{\widehat{\Omega}_t(z)}|\nabla w|^2dx\leq\frac{C}{(s-t)^2}\int_{\Omega_{s}(z')}|w|^2dx +C(s-t)^2\int_{\Omega_{s}(z')}|f|^2dx.
\end{align}

For $|z'|\leq R_{0}$, $\delta<s\leq\vartheta(\kappa_{1},\kappa_{3})\delta^{1/m}$, $\vartheta(\kappa_{1},\kappa_{3})=\frac{1}{2^{m+1}\kappa_{3}\max\{1,\kappa_{1}^{1/m-1}\}}$, it follows from conditions ({\bf{A1}}) and ({\bf{A2}}) that for  $(x',x_{d})\in\Omega_{s}(z')$,
\begin{align}\label{KHW01}
|\delta(x')-\delta(z')|\leq&|h_{1}(x')-h_{1}(z')|+|h_{2}(x')-h_{2}(z')|\notag\\
\leq&(|\nabla_{x'}h_{1}(x'_{\theta_{1}})|+|\nabla_{x'}h_{2}(x'_{\theta_{2}})|)|x'-z'|\notag\\
\leq&\kappa_{3}|x'-z'|(|x'_{\theta_{1}}|^{m-1}+|x'_{\theta_{2}}|^{m-1})\notag\\
\leq&2^{m-1}\kappa_{3}s(s^{m-1}+|z'|^{m-1})\notag\\
\leq&\frac{\delta(z')}{2}.
\end{align}
Then, we get
\begin{align*}
\frac{1}{2}\delta(z')\leq\delta(x')\leq\frac{3}{2}\delta(z'),\quad\mathrm{in}\;\Omega_{s}(z').
\end{align*}
This, together with the fact that $w=0$ on $\Gamma^{\pm}_{2R_{0}}$, yields that
\begin{align}\label{ADE007}
\int_{\Omega_{s}(z')}|w|^{2}\leq C\delta^{2}\int_{\Omega_{s}(z')}|\nabla w|^{2},
\end{align}
and, in light of \eqref{KHQ010},
\begin{align}\label{ADE006}
\int_{\Omega_{s}(z')}|f|^2dx\leq& C\Theta^{2}(\varphi,\psi)(z')\int_{\Omega_{s}(z')}\frac{1}{\delta^{2}}dx+C(\|\varphi\|^{2}_{C^{2}}+\|\psi\|^{2}_{C^{2}})\int_{\Omega_{s}(z')}\frac{|x'-z'|^{2}}{\delta^{2}}dx\notag\\
&+Cs^{n-1}\delta(\|\varphi\|^{2}_{C^{2}}+\|\psi\|^{2}_{C^{2}})\notag\\
\leq&C\Theta^{2}(\varphi,\psi)(z')s^{n-1}\delta^{-1}+Cs^{n+1}\delta^{-1}(\|\varphi\|^{2}_{C^{2}}+\|\psi\|^{2}_{C^{2}}).
\end{align}
Denote
\begin{align*}
F(t):=\int_{\Omega_{t}(z')}|\nabla w|^{2}.
\end{align*}
Then we know from (\ref{ADE007})--(\ref{ADE006}) that
\begin{align}\label{ADE008}
F(t)\leq &\left(\frac{c\delta}{s-t}\right)^2F(s)\notag\\
&+C(s-t)^{2}s^{n-1}\delta^{-1}\big(\Theta^{2}(\varphi,\psi)(z')+s^{2}(\|\varphi\|^{2}_{C^{2}}+\|\psi\|^{2}_{C^{2}})\big),
\end{align}
where $c$ and $C$ are universal constants.

Choose $k=\left[\frac{\vartheta(\kappa_{1},\kappa_{3})}{4c\delta^{1-1/m}}\right]+1$ and $t_{i}=\delta+2ci\delta,\;i=0,1,2,...,k$. Then, (\ref{ADE008}), together with $s=t_{i+1}$ and $t=t_{i}$, reads that
$$F(t_{i})\leq\frac{1}{4}F(t_{i+1})+C(i+1)^{n+1}\delta^{n}\big(\Theta^{2}(\varphi,\psi)(z')+\delta^{2}(\|\varphi\|^{2}_{C^{2}}+\|\psi\|^{2}_{C^{2}})\big).$$
In view of the fact that $\varphi\neq0$ on $\Gamma^{+}_{2R_{0}}$ or $\psi\neq0$ on $\Gamma^{-}_{2R_{0}}$, it follows from $k$ iterations and (\ref{lem2.2equ}) that for a sufficiently small $\varepsilon>0$,
\begin{align*}
F(t_{0})\leq C\delta^{n}\big(\Theta^{2}(\varphi,\psi)(z')+\delta^{2}(\|\varphi\|^{2}_{C^{2}}+\|\psi\|^{2}_{C^{2}})\big).
\end{align*}
\end{proof}

\begin{proof}[The proof of Theorem \ref{thm1.1}]
Making a change of variables in $\Omega_{\delta}(z')$ as follows:
\begin{align*}
\begin{cases}
x'-z'=\delta y',\\
x_{n}=\delta y_{n},
\end{cases}
\end{align*}
we turn $\Omega_{\delta}(z')$ into $Q_{1}$, where, for $0<r\leq 1$,
\begin{align*}
Q_{r}=\left\{y\in\mathbb{R}^{n}\,\Big|\,\frac{1}{\delta}h_{2}(\delta y'+z')<y_{n}<\frac{\varepsilon}{\delta}+\frac{1}{\delta}h_{1}(\delta y'+z'),\;|y'|<r\right\},
\end{align*}
with its top and bottom boundaries denoted, respectively, by
\begin{align*}
\widehat{\Gamma}^{+}_{r}=&\left\{y\in\mathbb{R}^{n}\,\Big|\,y_{n}=\frac{\varepsilon}{\delta}+\frac{1}{\delta}h_{1}(\delta y'+z'),\;|y'|<r\right\},
\end{align*}
and
\begin{align*}
\widehat{\Gamma}^{-}_{r}=&\left\{y\in\mathbb{R}^{n}\,\Big|\,y_{n}=\frac{1}{\delta}h_{2}(\delta y'+z'),\;|y'|<r\right\}.
\end{align*}
Similar to \eqref{KHW01}, we obtain that for $x\in\Omega_{\delta}(z')$,
\begin{align*}
|\delta(x')-\delta(z')|
\leq&2^{m-1}\kappa_{3}\delta(\delta^{m-1}+|z'|^{m-1})\notag\\
\leq&2^{m}\kappa_{3}\max\{1,\kappa_{1}^{1/m-1}\}\delta^{2-1/m}.
\end{align*}
This implies that
\begin{align}\label{KHW03}
\left|\frac{\delta(x')}{\delta(z')}-1\right|\leq2^{m}\max\{1,\kappa_{1}^{1/m-1}\}\kappa_{2}^{1-1/m}\kappa_{3}R_{0}^{m-1},
\end{align}
which, in combination with the fact that $R_{0}$ is a small positive constant, reads that $Q_{1}$ is of nearly unit size as far as applications of Sobolev embedding theorems and classical $L^{p}$ estimates for elliptic systems are concerned.

Let \begin{align*}
%\hat{u}_{l}(y', y_n)=\bar{u}_{l}(\delta y'+z', \delta y_n),\qquad
\hat{w}_{l}(y', y_n)=w_{l}(\delta y'+z', \delta y_n).\end{align*}
 Therefore, $\hat{w}_{l}(y)$ solves
 \begin{align*}
\begin{cases}
  \partial_\alpha\left(\hat{A}_{ij}^{\alpha\beta}\partial_\beta \hat{w}_{l}^j+\hat{B}_{ij}^\alpha \hat{w}_{l}^j\right)+\hat{C}_{ij}^\beta \partial_\beta \hat{w}_{l}^j+\hat{D}_{ij}\hat{w}_{l}^j=\hat{f}_{l}^{i},\quad
&\hbox{in}\  Q_1,  \\
\hat{w}_{l}=0, \quad&\hbox{on} \ \widehat{\Gamma}_1^\pm,
\end{cases}
\end{align*}
where
\begin{equation}\label{hatABCD}
\begin{aligned}
\hat{A}(y)=&A(\delta y'+z', \delta y_n),\quad\ \ \hat{B}(y)=\delta\,B(\delta y'+z', \delta y_n), \\
\hat{C}(y)=&\delta\,C(\delta y'+z', \delta y_n),\quad \hat{D}(y)=\delta^{2}D(\delta y'+z', \delta y_n),
\end{aligned}
\end{equation}
and $\hat{f}_{l}^{i}:=-\partial_\alpha\left(\hat{A}_{ij}^{\alpha\beta}\partial_\beta \hat{u}_{l}^j+\hat{B}_{ij}^\alpha \hat{u}_{l}^j\right)-\hat{C}_{ij}^\beta \partial_\beta \hat{u}_{l}^j-\hat{D}_{ij}\hat{u}_{l}^j$.

In view of $\hat{w}=0$ on the upper and lower boundaries of $Q_1$, we have, by the Poincar\'{e} inequality, that
$$\|\hat{w}_{l}\|_{H^1(Q_1)}\leq C\|\nabla \hat{w}_{l}\|_{L^2(Q_1)}.$$
 Using the Sobolev embedding theorem and classical $W^{2, p}$ estimates for elliptic systems, %(see e.g. \cite{AD2}, or Theorem 2.5 in \cite{G}),
  we have, for some $p>n$,
 $$\|\nabla \hat{w}_{l}\|_{L^\infty(Q_{1/2})}\leq C\|\hat{w}_{l}\|_{W^{2, p}(Q_{1/2})}\leq C\left(\|\nabla \hat{w}_{l}\|_{L^2(Q_1)}+\|\hat{f}_{l}\|_{L^\infty(Q_1)}\right).$$
Since
$$\|\nabla \hat{w}_{l}\|_{L^\infty(Q_{1/2})}=\delta\|\nabla w_{l}\|_{L^\infty(\widehat{\Omega}_{\delta/2}(z'))},\quad\|\nabla\hat{w}_{l}\|_{L^2(Q_1)}=\delta^{1-\frac{n}{2}}\|\nabla w_{l}\|_{L^2(\Omega_\delta(z'))},$$and$$\|\hat{f}_{l}\|_{L^\infty(Q_1)}=\delta^2\|f_{l}\|_{L^\infty(\Omega_\delta(z'))},$$
tracing back to $w$ through the transforms, we have
\begin{align}\label{AZG01}
\|\nabla w_{l}\|_{L^\infty(\Omega_{\delta/2}(z'))}\leq\frac{C}{\delta}\left(\delta^{1-\frac{n}{2}}\|\nabla w_{l}\|_{L^2(\Omega_\delta(z'))}+\delta^2\|f_{l}\|_{L^\infty(\Omega_\delta(z'))}\right).
\end{align}
Utilizing \eqref{KHQ010} and \eqref{KHW03}, we obtain that for $x\in\Omega_{\delta}(z')$,
\begin{align*}
\delta(z')|f_{l}(x)|\leq&C\Theta(\varphi,\psi)(z')+C\delta(z')\bigg(\frac{|x'-z'|}{\delta(z')}+1\bigg)\big(\|\varphi^{l}\|_{C^{2}(\Gamma^{+}_{2R_{0}})}+\|\psi^{l}\|_{C^{2}(\Gamma^{-}_{2R_{0}})}\big)\notag\\
\leq&C\Theta(\varphi,\psi)(z')+C\delta(z')\big(\|\varphi^{l}\|_{C^{2}(\Gamma^{+}_{2R_{0}})}+\|\psi^{l}\|_{C^{2}(\Gamma^{-}_{2R_{0}})}\big).
\end{align*}
This, together with \eqref{AVR001} and \eqref{AZG01}, yields that for $z\in\Omega_{R_{0}}$,
\begin{align*}
|\nabla w_{l}(z)|\leq& C\Theta(\varphi,\psi)(z')+C\delta(z')\big(\|\varphi^{l}\|_{C^{2}(\Gamma^{+}_{2R_{0}})}+\|\psi^{l}\|_{C^{2}(\Gamma^{-}_{2R_{0}})}\big).
\end{align*}
Therefore, in light of decompositions \eqref{split001} and \eqref{split002}, we derive
\begin{align*}
|\nabla(u-\bar{u})(z)|\leq& C\Theta(\varphi,\psi)(z')+C\delta(z')\big(\|\varphi\|_{C^{2}(\Gamma^{+}_{2R_{0}})}+\|\psi\|_{C^{2}(\Gamma^{-}_{2R_{0}})}\big).
\end{align*}
That is, Theorem \ref{thm1.1} is established.

\end{proof}

\section{Proof of Theorem \ref{thm003}}\label{sc3}
Similar to the proof of Lemma \ref{lem2.1} with a slight modification, we have
\begin{lemma}\label{lem2.111}
Assume as above. Let $u\in H^{1}(\Omega_{2R_{0}}; \mathbb{R}^N)$ be a weak solution of problem (\ref{eq1.1}), then
\begin{align}\label{lem2.2equu93789}
\|\nabla u\|_{L^{2}(\Omega_{R_{0}})}\leq C\|u\|_{L^2(\Omega_{2R_{0}})},
\end{align}
where $C$ depends on $n$, $\lambda$, $R_{0}$, $\kappa_1$, $\kappa_2$, $\kappa_{3}$ and $\kappa_{4}$, but not on $\varepsilon$.
\end{lemma}
\begin{proof}[The proof of Theorem \ref{thm003}]
Following the same argument as in \eqref{KTA001}, we have the iteration formula as follows: for $0<t<s<R_{0}$,
\begin{align}\label{zzwo01}
\int_{\Omega_{t}(z')}|\nabla u|^{2}dx\leq\frac{C}{(s-t)^{2}}\int_{\Omega_{s}(z')}|u|^{2}dx.
\end{align}
Analogously as before, we know that for $|z'|\leq R_{0}$, $\delta<s\leq\vartheta(\kappa_{1},\kappa_{3})\delta^{1/m}$, $\vartheta(\kappa_{1},\kappa_{3})=\frac{1}{2^{m+1}\kappa_{3}\max\{1,\kappa_{1}^{1/m-1}\}}$,
\begin{align*}
\frac{1}{2}\delta(z')\leq\delta(x')\leq\frac{3}{2}\delta(z'),\quad\mathrm{in}\;\Omega_{s}(z').
\end{align*}
Then combining the fact that $u=0$ on $\Gamma^{\pm}_{2R_{0}}$, we deduce
\begin{align*}
\int_{\Omega_{s}(z')}|u|^{2}\leq C\delta^{2}\int_{\Omega_{s}(z')}|\nabla u|^{2}.
\end{align*}
Then the iteration formula \eqref{zzwo01} becomes
\begin{align*}
\int_{\Omega_{t}(z')}|\nabla u|^{2}dx\leq\left(\frac{C\delta}{s-t}\right)^{2}\int_{\Omega_{s}(z')}|\nabla u|^{2}dx.
\end{align*}
Define
\begin{align*}
G(t):=\int_{\Omega_{t}(z')}|\nabla u|^{2}.
\end{align*}
Then we have
\begin{align}\label{ADE008001}
G(t)\leq &\left(\frac{c_{1}\delta}{s-t}\right)^2G(s),
\end{align}
where $c_{1}$ is a positive constant independent of $\varepsilon$.

Analogously as above, pick $k=\left[\frac{\vartheta(\kappa_{1},\kappa_{3})}{4c_{1}\delta^{1-1/m}}\right]+1$ and $t_{i}=\delta+2c_{1}i\delta,\;i=0,1,2,...,k$. Then, (\ref{ADE008001}), together with $s=t_{i+1}$ and $t=t_{i}$, yields that
$$G(t_{i})\leq\frac{1}{4}G(t_{i+1}).$$
After $k$ iterations, it follows from \eqref{lem2.2equu93789} that
\begin{align}\label{ADE010PSG91}
G(t_{0})=\int_{\Omega_{\delta}(z')}|\nabla u|^{2}\leq\frac{1}{4^{k}}G(t_{k})\leq Ce^{-\frac{1}{C\delta^{1-1/m}}}\|u\|^{2}_{L^{2}(\Omega_{2R_{0}})}.
\end{align}

Let
\begin{align*}
\hat{u}(y', y_n)=u(\delta y'+z', \delta y_n).
\end{align*}
Therefore, $\hat{u}(y)$ solves
\begin{align*}
\begin{cases}
\partial_\alpha\left(\hat{A}_{ij}^{\alpha\beta}\partial_\beta \hat{u}^j+\hat{B}_{ij}^\alpha \hat{u}^j\right)+\hat{C}_{ij}^\beta \partial_\beta \hat{u}^j+\hat{D}_{ij}\hat{u}^j=0,\quad
&\hbox{in}\  Q_1,  \\
\hat{u}=0, \quad&\hbox{on} \ \widehat{\Gamma}_1^\pm,
\end{cases}
\end{align*}
where the coefficients $\hat{A}(y)$, $\hat{B}(y)$, $\hat{C}(y)$ and $\hat{D}(y)$ are defined by \eqref{hatABCD}.

In view of $\hat{u}=0$ on $\widehat{\Gamma}_1^\pm$, by utilizing the Poincar\'{e} inequality, the Sobolev embedding theorem and classical $W^{2, p}$ estimates for elliptic systems again, we obtain that for some $p>n$,
$$\|\nabla \hat{u}\|_{L^\infty(Q_{1/2})}\leq C\|\hat{u}\|_{W^{2, p}(Q_{1/2})}\leq C\|\nabla \hat{u}\|_{L^2(Q_1)}.$$
Then back to $u$, it follows from \eqref{ADE010PSG91} that
\begin{align*}
\|\nabla u\|_{L^\infty(\Omega_{\delta/2}(z'))}\leq C\delta^{-\frac{n}{2}}\|\nabla u\|_{L^2(\Omega_\delta(z'))}\leq C\delta^{-\frac{n}{2}}e^{-\frac{1}{2C\delta^{1-1/m}}}\|u\|_{L^{2}(\Omega_{2R_{0}})},
\end{align*}
which implies that for $z\in\Omega_{R_{0}}$,
\begin{align*}
|\nabla u(z)|\leq C\delta^{-\frac{n}{2}}e^{-\frac{1}{C\delta^{1-1/m}}}\|u\|_{L^{2}(\Omega_{2R_{0}})}.
\end{align*}
\end{proof}

\section{Application in the Lam\'{e} systems}\label{sc56}
Recall that a system is called a system of elasticity if $N=n$, the coefficients satisfy
$$A_{ij}^{\alpha\beta}(x)=A_{ji}^{\beta\alpha}(x)=A_{\alpha\,j}^{i\beta}(x),$$
and the following ellipticity condition: for all $n\times{n}$ symmetric matrices $\xi_{\alpha}^{i}$,
\begin{align}\label{Symmetrypositive09}
\tau_{1}|\xi|\leq\,A_{ij}^{\alpha\beta}(x)\xi_{\alpha}^{i}\xi_{\beta}^{j}\leq\tau_{2}|\xi|^{2},\quad 0<\tau_{1}<\tau_{2}<\infty.
\end{align}
In particular, for $i, j, k, l=1,2,...,n$, we set
\begin{align}\label{ZWZL001}
A_{ij}^{\alpha\beta}=\lambda\delta_{i\alpha}\delta_{j\beta}+\mu(\delta_{i\beta}\delta_{\alpha j}+\delta_{ij}\delta_{\alpha\beta}),\quad B_{ij}^{\alpha}=C_{ij}^{\beta}=D_{ij}=0,
\end{align}
where $\lambda$ and $\mu$ are Lam\'{e} constants satisfying $\mu>0$ and $n\lambda+2\mu>0$, and $\chi_{ij}$ is the kronecker symbol: $\chi_{ij}=0$ for $i\neq j$, $\chi_{ij}=1$ for $i=j$. Thus $A_{ij}^{\alpha\beta}$ satisfies the ellipticity condition \eqref{Symmetrypositive09} with $\tau_{1}=\min\{2\mu,n\lambda+2\mu\}$ and $\tau_{2}=\max\{2\mu,n\lambda+2\mu\}$.

Then problem \eqref{eq1.1} becomes the boundary value problem for the Lam\'{e} systems as follows:
\begin{align}\label{zzw901}
\begin{cases}
\mu\Delta{u}+(\lambda+\mu)\nabla(\nabla\cdot{u})=0,\quad&
\hbox{in}\  \Omega_{2R_{0}},  \\
u=\varphi(x),\quad &\hbox{on}\ \Gamma_{2R_{0}}^+,\\
u=\psi(x), \quad&\hbox{on} \ \Gamma_{2R_{0}}^-.
\end{cases}
\end{align}
In view of \eqref{ZWZL001}, it follows from a straightforward computation that for $i=1,2,...,n$,
\begin{align}\label{znm001}
A_{ij}^{nn}=&
(\lambda+\mu)\chi_{in}\chi_{jn}+\mu\chi_{ij},\quad l=1,2,...,n,
\end{align}
and
\begin{align}\label{znm002}
\sum^{n-1}_{\alpha=1}A_{il}^{\alpha n}\partial_{\alpha}\delta+\sum^{n-1}_{\beta=1}A_{il}^{n\beta}\partial_{\beta}\delta=
\begin{cases}
(\lambda+\mu)\chi_{in}\partial_{l}\delta,&l=1,...,n-1,\\
(\lambda+\mu)\partial_{i}\delta,&l=n.
\end{cases}
\end{align}
In view of \eqref{znm001}, we see that the matrix $(A_{ij}^{nn})_{N\times N}$ satisfies the uniform elliptic condition \eqref{QDZAQ} with $\Lambda_{1}=\mu$ and $\Lambda_{2}=\lambda+2\mu$. Then substituting \eqref{znm001}--\eqref{znm002} into \eqref{ZZW666}, it follows from the Cramer's rule that
\begin{align}\label{znm003}
\mathcal{G}_{l}=\frac{\lambda+\mu}{\lambda+2\mu}(\varphi^{l}(x',\varepsilon+h_{1}(x'))-\psi^{l}(x',h_{2}(x')))\partial_{l}\delta\, e_{n},\quad l=1,...,n-1,
\end{align}
and
\begin{align}\label{znm005}
\mathcal{G}_{n}=\frac{\lambda+\mu}{\mu}(\varphi^{n}(x',\varepsilon+h_{1}(x'))-\psi^{n}(x',h_{2}(x')))\sum^{n-1}_{l=1}\partial_{l}\delta\,e_{l}.
\end{align}
Inserting \eqref{znm003}--\eqref{znm005} to \eqref{OPQ}, the leading term for the gradient of a solution to the Lam\'{e} systems is expressed as follows: for $(x',x_{n})\in\Omega_{2R_{0}}$,
\begin{align}\label{OKL01}
\bar{u}_{\lambda,\mu}(x',x_{n}):=&\varphi(x',\varepsilon+h_{1}(x'))\bar{v}+\psi(x',h_{2}(x'))(1-\bar{v})\notag\\
&+\frac{\lambda+\mu}{\mu}\mathfrak{r}(\bar{v})(\varphi^{n}(x',\varepsilon+h_{1}(x'))-\psi^{n}(x',h_{2}(x')))\sum^{n-1}_{l=1}\partial_{l}\delta\,e_{l}\notag\\
&+\frac{\lambda+\mu}{\lambda+2\mu}\mathfrak{r}(\bar{v})\sum^{n-1}_{l=1}\partial_{l}\delta(\varphi^{l}(x',\varepsilon+h_{1}(x'))-\psi^{l}(x',h_{2}(x')))\,e_{n}.
\end{align}
For $x\in\Omega_{R_{0}}$, define
\begin{align}\label{QWA01999}
\bar{\Theta}_{\delta}(\varphi,\psi)(x'):=&|\varphi(x',\varepsilon+h_{1}(x'))-\psi(x',h_{2}(x'))|\delta^{1-2/m}\notag\\
&+|\nabla_{x'}(\varphi(x',\varepsilon+h_{1}(x'))-\psi(x',h_{2}(x')))|.
\end{align}
A direct application of Theorem \ref{thm1.1} yields that
\begin{corollary}\label{coro001}
Assume that hypotheses $\rm{(}${\bf{A1}}$\rm{)}$--$\rm{(}${\bf{A3}}$\rm{)}$ hold, and $\varphi\neq0$ on $\Gamma^{+}_{2R_{0}}$ or $\psi\neq0$ on $\Gamma^{-}_{2R_{0}}$. Let $u\in H^{1}(\Omega_{2R_{0}}; \mathbb{R}^N)$ be a weak solution of problem \eqref{zzw901}. Then, for a sufficiently small $\varepsilon>0$, $x\in \Omega_{R_{0}}$,
\begin{align*}
\nabla u=&\nabla\bar{u}_{\lambda,\mu}+O(1)\Big(\bar{\Theta}_{\delta}(\varphi,\psi)(x')+\delta\big(\|\varphi\|_{C^{2}(\Gamma^{+}_{2R_{0}})}+\|\psi\|_{C^{2}(\Gamma^{-}_{2R_{0}})}\big)\Big),
\end{align*}
where the leading term $\bar{u}_{\lambda,\mu}$ is given by \eqref{OKL01} and the rest term $\bar{\Theta}_{\delta}(\varphi,\psi)(x')$ is defined by \eqref{QWA01999}.
\end{corollary}
\begin{remark}
We see from Corollary \ref{coro001} that for the Lam\'{e} systems \eqref{zzw901}, we can capture a smaller residual part $\bar{\Theta}_{\delta}(\varphi,\psi)(x')$ than $\Theta(\varphi,\psi)(x')$ defined in \eqref{QWA019} for general elliptic systems \eqref{eq1.1}, see also Theorem 3.1 of \cite{MZ2021} for detailed proof. In fact, for the Lam\'{e} systems, due to the fact that $\partial_{\alpha}A_{ij}^{\alpha\beta}=B_{ij}^{\alpha}=C_{ij}^{\beta}=D_{ij}=0$, then there is no lower order term appearing in \eqref{ZWPLN001} so that we obtain a more small upper bound estimate of $f_{l}^{i}$ than that of \eqref{KHQ010} as follows:
\begin{align*}
|f^{i}_{l}|\leq&C\bar{\Theta}_{\delta}(\varphi,\psi)(x')\delta^{-1}+C(\|\varphi^{l}\|_{C^{2}}+\|\psi^{l}\|_{C^{2}}).
\end{align*}

\end{remark}
\begin{remark}
For $l=1,2,...,n$, if we let $\varphi=e_{l}$ on $\Gamma^{+}_{2R_{0}}$ and $\psi=0$ on $\Gamma^{-}_{2R_{0}}$ for problem \eqref{zzw901} , then the leading term \eqref{OKL01} becomes
\begin{align*}
\bar{u}_{\lambda,\mu}=\big(\bar{u}_{\lambda,\mu}^{1},...,\bar{u}_{\lambda,\mu}^{n}\big),\quad\bar{u}_{\lambda,\mu}^{l}=&
\begin{cases}
\bar{v}\,e_{l}+\frac{\lambda+\mu}{\lambda+2\mu}\mathfrak{r}(\bar{v})\partial_{l}\delta\,e_{n},&l=1,...,n-1,\\
\bar{v}\,e_{n}+\frac{\lambda+\mu}{\mu}\mathfrak{r}(\bar{v})\sum_{i=1}^{n-1}\partial_{i}\delta\,e_{i},&l=n,
\end{cases}
\end{align*}
which was previously captured in \cite{LX2020}.
\end{remark}

\noindent{\bf{\large Acknowledgements.}} The author was partially supported by CPSF (2021M700358). This paper was initiated while the author was visiting the Bernoulli Institute for Mathematics, Computer Science and Artificial Intelligence at University of Groningen in Netherland from October 2019 to October 2020. The author thanks the hospitality and the stimulating environment.

\bibliographystyle{plain}

\def\cprime{$'$}

\end{document}